\documentclass[12pt,reqno]{amsart}
\usepackage{graphics}
\usepackage{latexsym}
\usepackage[dvips]{epsfig}
\usepackage{color}
\usepackage{graphicx}
\usepackage{amsmath}
\usepackage{amssymb}
\usepackage{amsfonts}
\usepackage{eufrak}
\usepackage{amstext}
\usepackage{amsopn}
\usepackage{amsbsy}
\usepackage{amscd}
\usepackage{amsxtra}
\usepackage{amsthm}
\usepackage{bm}
\usepackage{CJK}

\newcommand{\R}{{\mathbb R}}

\newcommand{\beq}{\begin{equation}}
\newcommand{\eeq}{\end{equation}}
\newcommand{\ben}{\begin{eqnarray}}
\newcommand{\een}{\end{eqnarray}}
\newcommand{\beno}{\begin{eqnarray*}}
\newcommand{\eeno}{\end{eqnarray*}}

\newtheorem{thm}{Theorem}[section]
\newtheorem{defi}[thm]{Definition}
\newtheorem{lem}[thm]{Lemma}
\newtheorem{prop}[thm]{Proposition}
\newtheorem{coro}[thm]{Corollary}
\newtheorem{rmk}[thm]{Remark}

\allowdisplaybreaks

\begin{document}

\title[Serrin's overdetermined problem]{On Serrin's overdetermined problem  and a conjecture of Berestycki, Caffarelli and Nirenberg}

\author[K. Wang]{ Kelei Wang}
 \address{\noindent K. Wang-
 School of Mathematics and Statistics, Wuhan University, Wuhan 430072, China.}

\author[J. Wei]{Juncheng Wei}
\address{\noindent J. Wei -Department of Mathematics, University of British
Columbia, Vancouver, B.C., Canada, V6T 1Z2  }
\email{jcwei@math.ubc.ca}

\begin{abstract}
This paper concerns rigidity results to Serrin's overdetermined problem in an epigraph
$$ \left\{\begin{aligned}
&\Delta u+ f(u)=0,\ \ \ \mbox{in}\ \Omega=\{(x^\prime,x_n): x_n>\varphi (x^\prime)\},\\
&u>0,\ \ \ \mbox{in}\ \Omega,\\
&u=0,\ \ \ \mbox{on}\ \partial\Omega,\\
&|\nabla u|=const.,~~\mbox{on}~\partial\Omega.
                          \end{aligned} \right .
$$
We prove that up to isometry the epigraph must be an half space and
that the solution $u$ must be one-dimensional, provided that one of the
following assumptions are satisfied: either $n=2$; or $ \varphi $ is
globally Lipschitz, or $n \leq 8$ and $ \frac{\partial u}{\partial
x_n} >0$ in $\Omega$. In view of the counterexample constructed in \cite{DPW} in
dimensions $n\geq 9$ this result is optimal. This partially answers a conjecture of Berestycki, Caffarelli and Nirenberg \cite{BCN}.

\end{abstract}



\maketitle

\date{}

\section{Introduction}
\setcounter{equation}{0}

This paper is concerned with the one dimensional symmetry problem for
the Serrin's  overdetermined problem in an epigraph.  More precisely we consider solutions to the following overdetermined problem:
\begin{equation}\label{equation 00}
 \left\{\begin{aligned}
&\Delta u+ f(u)=0,\ \ \ \mbox{in}\ \Omega\\
&u>0,\ \ \ \mbox{in}\ \Omega,\\
&u=0,\ \ \ \mbox{on}\ \partial\Omega,\\
&\frac{\partial u}{\partial \nu}=const.,~~\mbox{on}~\partial\Omega
                          \end{aligned} \right.
\end{equation}
where $f$ is a Lipschitz nonlinearity, $ \nu$ is the outer normal at $\partial \Omega$, and $ \frac{\partial u}{\partial \nu}$ is a constant  which is not prescribed a priori.

\medskip

A classical result of  Serrin's \cite{Serrin} asserts that if $\Omega$ is a bounded and smooth domain for which there is a positive solution to the overdetermined equation (\ref{equation 00}) then $\Omega$ is a sphere and $u$ is radially symmetric.

In the analysis of blown up version of free boundary problem, it is natural also to consider Serrin's overdetermined problem in unbounded domains. (See Berestycki-Caffarelli-Nirenberg \cite{BCN0}.)   A natural class of unbounded domains to be considered are epigraphs, namely domains $\Omega$ of the form
  \begin{equation}
  \label{epi}
  \Omega = \{ x \in \R^{n} \ /\ x_N > \varphi (x^{'}) \}
  \end{equation}
  where $ x^{'}= (x_1, \ldots, x_{n-1})$ and  $\varphi: \R^{n-1}\to \R$ is a smooth function. In \cite{BCN},  Berestycki, Caffarelli and Nirenberg proved, under conditions on $f$
  that are satisfied for instance the Allen-Cahn nonlinearity below, the following result: If  $\varphi$ is uniformly Lipschitz and {\em asymptotically flat} at infinity, and Problem \eqref{equation 00} is solvable, then $\varphi$ must be a linear function, in other words $\Omega$ must be a {\em half-space}. This result was improved by Farina and Valdinoci \cite{F1}, by lifting the asymptotic flatness condition and smoothness on $f$, under the dimension constraint $n\leq 3$ and other assumptions (see remarks below). When the epigraph is  coercive (see (\ref{coercive}) below) they can also consider an arbitrary nonlinearity.

  \medskip
  In \cite[pp.1110]{BCN}, the following conjecture on Serrin's overdetermined problem in unbounded domains was raised.

\medskip

  \noindent
  {\bf Berestycki-Caffarelli-Nirenberg Conjecture:}  {\em Assume that   $ \Omega$  is  a smooth domain with $ \Omega^c$ connected and that there is a bounded positive solution of (\ref{equation 00}) for some Lipschitz function $f$ then $\Omega $ is  either a half-space, or a cylinder $\Omega = B_k\times \R^{n-k}$, where $B_k$ is a $k$-dimensional Euclidean ball, or  the complement of a ball or a cylinder.}

\medskip

In this paper we are mainly concerned with the BCN conjecture in the epigraph case (\ref{epi}), namely the following overdetermined problem
 \begin{equation}\label{equation 0}
 \left\{\begin{aligned}
&\Delta u+ f(u)=0,\  u>0\ \ \mbox{in}\ \Omega= \{ x_n >\varphi (x^{'}) \} \\
&u=0,\ \ \ \mbox{on}\ \ \{ x_n=\varphi (x^{'})\},\\
&\frac{\partial u}{\partial \nu}=const.,~~\mbox{on}~\ \ \{ x_n= \varphi (x^{'}) \}.
                          \end{aligned} \right.
\end{equation}

 In this  case,  the BCN conjecture states that if Serrin's problem (\ref{equation 0}) is solvable, then $\Omega$ must be an half-space.   In a recent paper, del Pino, Pacard and the second author \cite{DPW} constructed an epigraph, which is a perturbation of the Bombieri-De Giorgi-Giusti minimal graph,   such that problem (\ref{equation 0}) admits a solution.  This counterexample requires dimension $n \geq 9$.  It remains open if the BCN Conjecture is true in low dimensions $ n\leq 8$. In this paper we shall give an affirmative answer to this question.

Before we proceed, we introduce the assumptions on the nonlinearity.  Let $ W(u)= -\int_0^u f(s) ds$. We assume that  $W$ is a standard double well potential, that is, $W\in
C^2([0,+\infty))$, satisfying
\begin{itemize}
\item [W1)]{\em $W\geq0$, $W(1)=0$ and $W>0$ in $[0,1)$;}
\item [W2)] {\em for some $\gamma\in(0,1)$, $W^\prime<0$ on
$(\gamma,1)$;}
\item [W3)] {\em there exists a constant $\kappa>0$, $W^{\prime\prime}\geq\kappa>0$ for all
$x\geq\gamma$;}
\item [W4)] {\em there exists a constant $p>1$, $W^\prime(u)\geq c(u-1)^p$ for
$u>1$.}
\end{itemize}

Moreover, we also assume that $W$ satisfies
\begin{enumerate}
\item [W5)] {\em $W^\prime<0$ in $(0,1)$, and either $W^\prime(0)\neq 0$ or
\[W^\prime(0)=0\quad \mbox{and}\quad W^{\prime\prime}(0)\neq 0.\]}
\end{enumerate}

A prototype example is $W(u)=(1-u^2)^2/4$ which gives the  Allen-Cahn
equation.

Under hypothesis (W1-4), there exists a unique function $g$ satisfying
\begin{equation}\label{1d profile}
 \left\{\begin{aligned}
&g^{\prime\prime}=W^\prime(g),\ \ \ \mbox{on}\ [0,+\infty),\\
&g(0)=0,\ \ \ \lim_{t\to+\infty}g(t)=1.
                          \end{aligned} \right .
\end{equation}
Moreover, $g$ has the following first integral:
\begin{equation}\label{first integral}
g^\prime(t)=\sqrt{2W(g(t))}>0,\quad \mbox{on}\ [0,+\infty).
\end{equation}
As $t\to+\infty$, $g(t)$ converges to $1$ exponentially. Hence the
following quantity is finite:
\[\sigma_0:=\int_0^{+\infty}\frac{1}{2}\big|g^\prime(t)\big|^2+W(g(t))dt <+\infty.\]

From now on we always assume that $W $ satisfies (W1-5).

\medskip

Our first result proves the conjecture in dimension 2 for any epigraph.

\begin{thm} \label{main result 4}
Let $ n=2$ and $W$ satisfy (W1-5).  If Serrin's overdetermined problem (\ref{equation 0}) has a solution then $ \Omega=\{ x_n >\varphi (x^{'}) \}$ must be a half space and up to isometry  $u(x)\equiv g(x\cdot e)$
for some unit vector $e$.
\end{thm}

Our second result proves the conjecture in all dimensions for any
Lipschitz or coercive graph.

\begin{thm}
\label{main result 3} Assume that  $\varphi$ is globally Lipschitz.
If Serrin's overdetermined problem \eqref{equation 0} has a solution
then $ \Omega=\{ x_n >\varphi (x^{'}) \}$ must be a half space and
up to isometry  $u(x)\equiv g(x\cdot e)$ for some unit vector $e$.
\end{thm}

\begin{thm}
\label{main result 5} Assume that  $\varphi$ is coercive, i.e.
\begin{equation}
\label{coercive}
\lim_{x^\prime\to\infty}\varphi(x^\prime)=+\infty.
\end{equation}
Then there is no solution to Serrin's overdetermined problem
\eqref{equation 0} in $ \Omega=\{ x_n >\varphi (x^{'}) \}$.
\end{thm}

The last result proves the conjecture in dimensions $n\leq 8$, under an additional assumption.

\begin{thm}\label{main result}
Let $u$ be a solution of \eqref{equation 0} satisfying the following monotonicity assumption in one direction
\begin{equation}\label{monotinicity in one direction}
\frac{\partial u}{\partial x_n}>0,\quad \mbox{in}\ \Omega.
\end{equation}
If $n\leq 8$  and $0\in\partial\Omega$, then $u(x)\equiv g(x\cdot
e)$ for some unit vector $e$ and $\Omega$ is an half space.
\end{thm}

We compare the results of this paper with those in the existing literature. Theorem \ref{main result 4} was proved by Farina and Valdinoci \cite{F1} under the assumption that the epigraph is globally Lipschitz (and for more general $f$).  They also proved Theorem \ref{main result 3} under the dimension restriction $n=2$ or $3$.  In view of the counterexample constructed by del Pino, Pacard and the second author \cite{DPW}, the dimension restriction in Theorem \ref{main result} is optimal. (We remark that the solutions constructed  in \cite{DPW}  also satisfy (\ref{monotinicity in one direction}).)

The extra condition (\ref{monotinicity in one direction}) in Theorem \ref{main result} is a natural one. See \cite{F1, F3}. This condition is always satisfied if the epigraph is globally Lipschitz  or coercive (\cite{BCN}). It seems that the monotonicity condition \eqref{monotinicity in one
direction} should follow from our other assumptions. However, this
is not clear at present.
It will be an interesting question to remove or prove this condition in general setting.

Theorems \ref{main result 4}-\ref{main result} have analogues in De Giorgi Conjecture for Allen-Cahn equation
\begin{equation}
\label{AC}
\Delta u + u-u^3=0 \  \ \ \mbox{in} \ \R^n
\end{equation}
which asserts that the only solution which is monotone in one direction must be one-dimensional.  Caffarelli-Cordoba  \cite{CC} proved the one-dimensional symmetry result under the assumption that
the level set is globally Lipschitz. (This corresponds to Theorem \ref{main result 3}.) De Giorgi's conjecture has been proven to be true for $n=2$ by Ghoussoub and Gui in \cite{gg}, for $n=3$ by Ambrosio and Cabre in \cite{AmC} and for $4\leq n\leq 8$ by Savin in \cite{Savin}, under the additional assumption that
$$
\lim_{x_{n}\to \pm \infty}u(x',x_{n} )=\pm 1.
$$
This conjecture was proven to be false for $n\geq 9$ by del Pino, Kowalczyk and Wei in \cite{dkwdg}. (Another proof of Savin's theorem is recently given by the first author \cite{Wang}.)

Now we explain the main ideas of our proof. The key observation is that under some conditions  (i.e. the monotonicity condition (\ref{monotinicity in one direction})) we shall prove that solutions to Serrin's overdetermined problem (\ref{equation 0}) are
 minimizers of the
functional
\begin{equation}\label{functional}
\int \frac{1}{2}|\nabla u|^2+W(u)\chi_{\{u>0\}}.
\end{equation}
(Here we only need $W$ to satisfy hypothesis (W1-4).) The Euler-Lagrange equation  corresponding to (\ref{functional}) ((\ref{equation}) below) is a one phase free boundary problem in which $ |\nabla u| = \sqrt{2W(0)}$ on the boundary.  To this end, we first establish
\begin{thm}\label{reduction to free boundary}
Let $u$ be a solution of \eqref{equation 0}, where $W$ satisfies
(W1-5). Then $|\nabla u|=\sqrt{2W(0)}$ on $\partial\Omega$.
\end{thm}
This is mainly because $\{u>0\}$ is an epigraph, we can touch
$\partial\{u>0\}$ by arbitrarily large balls from both sides. Then
we construct suitable comparisons in these balls to determine
$|\nabla u|\lfloor_{\partial\Omega}$. (Similar idea has been used in \cite{BCN}.) Theorem \ref{reduction to free boundary} does not hold for other unbounded domains. See examples of Delaunay type domains in \cite{DPW}.

\medskip

With hypothesis (W5) and the monotonicity condition
\eqref{monotinicity in one direction}, we further show that a
solution to \eqref{equation 0} is necessarily a minimizer of
\eqref{functional}.

 Hence the proof of
Theorem \ref{main result} is reduced to the study of solutions to
the following one phase free boundary problem:
\begin{equation}\label{equation}
 \left\{\begin{aligned}
&\Delta u=W^\prime(u),\ \ \ \mbox{in}\ \Omega=\{u>0\},\\
&u>0,\ \ \ \mbox{in}\ \Omega,\\
&u=0,\ \ \ \mbox{on}\ \partial\Omega,\\
&|\nabla u|=\sqrt{2W(0)},~~\mbox{on}~\partial\Omega.
                          \end{aligned} \right .
\end{equation}
In the general case, a solution $u$ to this equation is a stationary
critical points of \eqref{functional}.

For this one phase free boundary problem, we have
\begin{thm}\label{main result 2}
Let $u$ be a minimizer of \eqref{functional} in $\R^n$ with
$0\in\partial\Omega$. If one of the blowing down limit of $u$ is an
half space, then $u(x)\equiv g(x\cdot e)$ for some unit vector $e$.
\end{thm}

This one phase free boundary problem bears many similarities with
the Allen-Cahn equation. Hence previous methods used to prove De
Giorgi conjecture for Allen-Cahn equations (cf. Savin \cite{Savin})
can be employed to study the one dimensional symmetry of solutions
to \eqref{equation}. In this paper, we shall follow the first author's approach
in \cite{Wang}, which uses an energy type quantity, the {\em
excess}. To this aim, we also present the Huichinson-Tonegawa's
theory for the convergence of general stationary critical points,
see Section 3.

\medskip

Finally we discuss other related progress made at the BCN conjecture. The conjecture, in the case of cylindrical domains, was disproved by Sicbaldi in \cite{Sic}, where he provided a counterexample in the case when $n\ge 3$ and $f(t)=\lambda t$, $\lambda >0$  by constructing a periodic perturbation of the cylinder $B^{n-1}\times R$ which supports a bounded solution to (\ref{equation 0}). In the two-dimensional case,  Hauswirth,  H\`elein and Pacard in \cite{pacard}  provided a counterexample in a strip-like domain for the case $f=0$. Explicitly, Serrin's overdetermined problem  is found to be solvable in the domain
$$ \Omega = \{ x\in \R^2 \ /\ |x_2| < \frac \pi 2 + \cosh (x_1) \},$$
where the solution found is unbounded.
Necessary geometric and topological conditions on $\Omega$ for solvability in the two-dimensional case have been found by Ros and Sicbaldi in \cite{ros}. The overdetermined problem in Riemannian manifolds has been considered by  Farina,  Mari and  Valdinoci in
\cite{F3}.

\medskip

This paper is organized as follows. In Section 2 we collect some
basic facts about the one phase free boundary problem, such as
Modica inequality and monotonicity formula. In Section 3 we present
 the Huichinson-Tonegawa
theory for the convergence of general stationary critical points.
Section 4 is devoted to prove Theorem \ref{main result 2}, following
\cite{Wang}. Most of these arguments are suitable adaption of
previous ones and we only state the results without proof. Only for
the integer multiplicity of the limit varifold in the
Huichinson-Tonegawa theory (Theorem \ref{quantization}), a new proof
is given, which follows the line introduced in Lin-Rivi\`{e}re
\cite{LR} and we think simplifies the existing methods. Section 5 is
devoted to proving that Theorem \ref{main result} can be reduced to
Theorem \ref{main result 2}.

\section{One phase free boundary problem}
\setcounter{equation}{0}

From this section to Section 4, $u$ always denotes a local minimizer
of \eqref{functional} in $\R^n$. We also assume that $u$ is
nontrivial and $0\in\partial\{u>0\}$.

We can show that $0\leq u\leq 1$ (see Proposition \ref{prop 2.1}
below) and it is Lipschitz continuous in $\R^n$ (see \cite{AC} and
\cite{CJK}). Hence $\{u>0\}$ is an open set, which we denote by
$\Omega$. Furthermore, by the partial regularity for free boundaries
in \cite{AC} and \cite{CJK}, $\partial\Omega$ is a smooth
hypersurface except a set of Hausdorff dimension at most $n-3$. The
last condition in \eqref{equation} is understood in the weak sense,
see \cite{AC}. At the smooth part of $\partial\Omega$, it also holds
in the classical sense.

\begin{prop}\label{prop 2.1}
$u\leq 1$ on $\R^n$.
\end{prop}
\begin{proof}
Following an idea of Brezis, first by the Kato inequality we can
show that
\[\Delta\left(u-1\right)_+\geq W^\prime(u)\chi_{\{u>1\}}\geq c\left(u-1\right)_+^p.\]
Then the claim follows from the Keller-Osserman theory.
\end{proof}
From this bound and the strong maximum principle, we can further
show that $u<1$ strictly in $\Omega$.

\begin{prop}[Modica inequality]\label{Modica
inequality}
\[\frac{1}{2}|\nabla u|^2\leq W(u),\quad \mbox{in}\ \Omega.\]
\end{prop}
\begin{proof}
Assume
\[\sup_{\Omega}\left(\frac{1}{2}|\nabla u|^2-W(u)\right)=:\delta>0,\]
and $x_i\in\Omega$ approaches this sup bound.

If $\limsup \mbox{dist}(x_i,\partial\Omega)>0$, we can argue as in
the proof of the usual Modica inequality (e.g. \cite{Modica}) to get
a contradiction.

If $\lim \mbox{dist}(x_i,\partial\Omega)=0$, then $u(x_i)\to0$.
Hence for all $i$ large,
\[\frac{1}{2}|\nabla u(x_i)|^2\geq W(0)+\frac{\delta}{2}.\]
Then we can proceed as in the proof of the gradient estimate for one
phase free boundary problem (cf. \cite[Corollary 6.5]{AC}) to get a
contradiction.
\end{proof}

\begin{rmk}\label{rmk 1}
The Modica inequality implies that $\partial\{u>0\}$ is mean convex
(see for example \cite{CJK}).
\end{rmk}

By considering domain variations, we can deduce the following
stationary condition:
\begin{equation}\label{stationary condition}
\int_{\Omega}\left(\frac{1}{2}|\nabla
u|^2+W(u)\chi_{\{u>0\}}\right)\mbox{div}X-DX(\nabla u,\nabla u)=0.
\end{equation}
As usual, this implies the following Pohozaev identity:
\begin{equation}
\int_{B_R}\frac{n-2}{2}|\nabla
u|^2+nW(u)\chi_{\{u>0\}}=R\int_{\partial B_R}\frac{|\nabla
u|^2}{2}-|u_r|^2+W(u)\chi_{\{u>0\}}.
\end{equation}
Together with the Modica inequality, this gives the following
monotonicity formula.
\begin{prop}[Monotonicity formula]\label{monotonicity formula}
\[E(r;u,x):=r^{1-n}\int_{B_r(x)\cap\Omega}\frac{1}{2}|\nabla u|^2+W(u)\chi_{\{u>0\}}\]
is non-decreasing in $r>0$.

Moreover,
\[\frac{d}{dr}E(r;u,x)=2r^{1-n}\int_{\partial B_r(x)}\Big|\nabla u(y)\cdot\frac{y-x}{|y-x|}\Big|^2
+r^{-n}\int_{B_r(x)}\left[W(u)\chi_\Omega-\frac{|\nabla
u|^2}{2}\right].\]
\end{prop}

\begin{coro}\label{clearing out 1}
Let $u$ be a non-trivial solution of \eqref{equation}. Then there
exists a constant $c>0$ such that, for any $x\in\partial\Omega$ and
$R>1$,
\[E(R;u,x)\geq cR^{n-1}.\]
\end{coro}
\begin{proof}
Because $x\in\partial\Omega$, by the non-degeneracy of $u$ near
$\partial\Omega$ (see \cite[Section 3]{AC}), there exists a
universal constant $c$ such that $|\Omega\cap B_1(x)|\geq c$. Then
because $|\nabla u|\leq C$, $W(u)\geq c$ in $\Omega\cap B_h(x)$ for
a universal constant $h$. This implies that $E(1;u,x)\geq c$ and the
claim follows from the monotonicity formula.
\end{proof}

On the other hand, for minimizers we have the following upper bound.
\begin{prop}\label{energy growth}
There exists a universal constant $C$ such that, for any $x\in\R^n$
and $R>1$,
\[\int_{B_R(x)}\frac{1}{2}|\nabla u|^2+W(u)\chi_{\{u>0\}}\leq CR^{n-1}.\]
\end{prop}
\begin{proof}
In $B_R(x)$, construct a comparison function in the following form:
\makeatletter
\let\@@@alph\@alph
\def\@alph#1{\ifcase#1\or \or $'$\or $''$\fi}\makeatother
\begin{equation}\label{1d solution}
{w(y)=}
\begin{cases}
1, &\mbox{in}\ B_{R-1}(x), \nonumber\\
|y-x|-R+1+\left(R-|y-x|\right)u(y)&\mbox{in}\ B_R(x)\setminus
B_{R-1}(x).
\end{cases}
\end{equation}
\makeatletter\let\@alph\@@@alph\makeatother

Note that $w>0$ in $B_R(x)$. A direct verification shows that
\[\int_{B_R(x)}\frac{1}{2}|\nabla w|^2+W(w)\leq CR^{n-1}.\]
The energy bound on $u$ follows from its minimality because $w=u$ on
$\partial B_R(x)$.
\end{proof}

\section{Huitchinson-Tonegawa theory}
\setcounter{equation}{0}

In this section we consider the convergence theory for general
stationary critical points of the functional
\begin{equation}\label{functional singular perturbed}
\int\frac{\varepsilon}{2}|\nabla
u_\varepsilon|^2+\frac{1}{\varepsilon}W(u_\varepsilon)\chi_{\{u_\varepsilon>0\}}.
\end{equation}

Let $u_\varepsilon$ be a sequence of stationary solutions in the
unit ball $B_1$, to the singularly perturbed problem
\begin{equation}\label{equation singular perturbed}
 \left\{\begin{aligned}
&\varepsilon\Delta u_\varepsilon=\frac{1}{\varepsilon}W^\prime(u_\varepsilon),\ \ \ \mbox{in}\ \{u_\varepsilon>0\},\\
&u_\varepsilon=0,\ \ \ \mbox{on}\ \partial\{u_\varepsilon>0\},\\
&|\nabla
u_\varepsilon|=\frac{1}{\varepsilon}\sqrt{2W(0)},~~\mbox{on}~\partial\{u_\varepsilon>0\}.
                          \end{aligned} \right .
\end{equation}
The stationary condition means that for any vector field $X\in
C_0^\infty(B_1,\R^n)$,
\begin{equation}\label{stationary condition singularly perturbed}
\int_{\Omega}\left(\frac{\varepsilon}{2}|\nabla
 u_\varepsilon|^2+\frac{1}{\varepsilon}W(u_\varepsilon)\chi_{\{u_\varepsilon>0\}}\right)\mbox{div}X
-\varepsilon DX(\nabla u_\varepsilon,\nabla u_\varepsilon)=0.
\end{equation}

We also assume that the energy of $u_\varepsilon$ is uniformly
bounded, that is,
\begin{equation}\label{energy bound}
\limsup_{\varepsilon\to0}\int_{B_1}\frac{\varepsilon}{2}|\nabla
u_\varepsilon|^2
+\frac{1}{\varepsilon}W(u_\varepsilon)\chi_{\{u_\varepsilon>0\}}<+\infty.
\end{equation}

Finally, to make the presentation simpler, we assume that $0\leq
u_\varepsilon\leq 1$ and it satisfies the Modica inequality
\begin{equation}\label{Modica inequality 2}
\frac{\varepsilon}{2}|\nabla u_\varepsilon|^2\leq
\frac{1}{\varepsilon}W(u_\varepsilon),\quad \mbox{in}\
\{u_\varepsilon>0\}.
\end{equation}
See \cite{H-T} for the general case, where two weaker conditions
(but sufficient for the application below) are derived from
\eqref{equation singular perturbed}-\eqref{energy bound}.

Of course, what is used in this paper is the following sequences
\[u_\varepsilon(x):=u(\varepsilon^{-1}x),\quad \varepsilon\to0,\]
where $u$ is a local minimizer of \eqref{functional} in $\R^n$. By
results in the previous section, they satisfy all of the above
assumptions.


We can assume that, up to a subsequence of $\varepsilon\to0$,
\[\varepsilon|\nabla
u_\varepsilon|^2dx\rightharpoonup\mu_1,\]
\[\frac{1}{\varepsilon}W(u_\varepsilon)dx\rightharpoonup\mu_2,\]
weakly as Radon measures, on any compact set of $B_1$.

A caution on our notation: in the following, unless otherwise
stated, $\varepsilon\to0$ means only a sequence $\varepsilon_i\to0$.

 In the following $\mu=\mu_1/2+\mu_2$
and $\Sigma=\mbox{spt}\mu$.

We can also assume the matrix valued measures
\[\varepsilon\nabla u_\varepsilon\otimes \nabla u_\varepsilon dx\rightharpoonup[\tau_{\alpha\beta}]\mu_1,\]
where $[\tau_{\alpha\beta}]$, $1\leq\alpha,\beta\leq n$, is
measurable with respect to $\mu_1$. Moreover, $\tau$ is nonnegative
definite $\mu_1$-almost everywhere and
\[\sum_{\alpha=1}^n\tau_{\alpha\alpha}=1,\quad \mu_1-a.e.\]

First, we need the following simple clearing out result, which is a
direct consequence of Corollary \ref{clearing out 1}.
\begin{prop}\label{clearing out}
There exists a universal constant $\eta$ small so that the following
holds. For any $r>0$, if
\[r^{1-n}\int_{B_r(x)}\frac{\varepsilon}{2}|\nabla u_\varepsilon|^2+\frac{1}{\varepsilon}W(u_\varepsilon)\leq\eta,\]
then either $u_\varepsilon\equiv 0$ in $B_{r/2}(x)$ or
$u_\varepsilon\geq 1-\gamma$.
\end{prop}

In the latter case of the previous lemma, we can improve the decay
estimate to an exponential one.
\begin{lem}\label{lem 3.2}
If $u_\varepsilon\geq 1-\gamma$ in $B_r(x)$, then
\[u_\varepsilon\geq 1-Ce^{-\frac{r}{C\varepsilon}}\quad \mbox{in}\ B_{r/2}(x).\]
\end{lem}
\begin{proof}
By (W3),
\[\Delta\left(1-u_\varepsilon\right)=-\frac{1}{\varepsilon^2}W^\prime(u_\varepsilon)
\geq\frac{c}{\varepsilon^2}\left(1-u_\varepsilon\right).\] From this
the decay estimate can be deduced, e.g. by a comparison with an
upper solution.
\end{proof}

Combining the monotonicity formula (Proposition \ref{monotonicity
formula}) and Proposition \ref{clearing out}, we get
\begin{lem}\label{lem 3.3}
For any $x\in\Sigma$,
\[\frac{1}{C}r^{n-1}\leq \mu(B_r(x))\leq Cr^{n-1},\]
for some universal constant $C$.
\end{lem}

Another consequence of Proposition \ref{clearing out} is:
\begin{lem}\label{lem 3.4}
On any connected compact set of $B_1\setminus\Sigma$, either
$u_\varepsilon\to1$ uniformly or $u_\varepsilon\equiv 0$ for all
$\varepsilon$ small.
\end{lem}
This is because for every $x$ not in $\Sigma$, there exists an $r>0$
such that $\mu(B_r(x))\leq \eta r^{n-1}/2$. Hence for all
$\varepsilon$ small,
\[\int_{B_r(x)}\frac{\varepsilon}{2}|\nabla u_\varepsilon|^2+W(u_\varepsilon)\chi_{\{u_\varepsilon>0\}}\leq
\eta r^{n-1},\] and Proposition \ref{clearing out} applies.

Similar to \cite{H-T}, by the Modica inequality (Proposition
\ref{Modica inequality}) and the monotonicity formula (Proposition
\ref{monotonicity formula}), we can show that
\begin{lem}
In $L^1_{loc}(B_1)$,
\[\frac{1}{\varepsilon}W(u_\varepsilon)\chi_{\{u_\varepsilon>0\}}-\frac{\varepsilon}{2}|\nabla u_\varepsilon|^2\to0.\]
\end{lem}
As a consequence, we have the following energy partition relation.
\begin{coro}
$\mu_1/2=\mu_2.$
\end{coro}

By passing to the limit in the monotonicity formula, we obtain the
corresponding monotonicity formula for the limit measure $\mu$.
\begin{lem} For any $x\in B_1$,
\[r^{1-n}\mu(B_r(x))\]
is non-decreasing in $r>0$. Moreover, for any $0<r_1<r_2<+\infty$,
\[r_2^{1-n}\mu(B_{r_2}(x))-r_1^{1-n}\mu(B_{r_1}(x))=2\int_{\Sigma\cap(B_{r_2}\setminus B_{r_1})}
\frac{\sum_{\alpha,\beta=1}^n\tau_{\alpha\beta}(y)(y-x)_\alpha(y-x)_\beta}{|x-y|^{n+1}}d\mu.\]
\end{lem}
By this lemma, we can define
\[\Theta(x):=\lim_{r\to0}\frac{\mu(B_r(x))}{r^{n-1}}.\]
This is an upper semi-continuous function. By Lemma \ref{lem 3.3},
$1/C\leq\Theta(x)\leq C$ everywhere on $\Sigma$.

Combining Proposition \ref{clearing out}, Lemma \ref{lem 3.2} and
Lemma \ref{lem 3.4}, we have the following characterization of
$\Sigma$.
\begin{coro}
$x\in\Sigma\Longleftrightarrow \Theta(x)>0\Longleftrightarrow
\Theta(x)\geq 1/C$.
\end{coro}

By the Preiss theorem \cite{Pre} (or by following the direct proof
in \cite{L}), we can show that
\begin{lem}
$\Sigma$ is countably $(n-1)$-rectifiable.
\end{lem}

By differentiation of Radon measures, the measure $\mu$ has the
following representation.
\begin{coro}
$\mu=\Theta\mathcal{H}^{n-1}\lfloor_{\Sigma}$.
\end{coro}

Next we show that
\begin{lem}\label{lem 3.11}
$I-\tau=T_x\Sigma, \ \mathcal{H}^{n-1}$-a.e. on $\Sigma$.
\end{lem}
This can be proved as in \cite{H-T}. However, here we would like to
give a new proof, which uses several ideas from \cite{LR}.

As in \cite{LR}, to prove this lemma, we only need to consider the
special case where $\Sigma=\R^{n-1}$.

Notation: $\mathcal{C}_1:= B_1^{n-1}\times(-1,1)$.
\begin{prop}\label{excess small}
If $\Sigma=\R^{n-1}$, then
\[\lim_{\varepsilon\to0}\int_{\mathcal{C}_1}\varepsilon\sum_{\alpha=1}^{n-1}
\left(\frac{\partial u_{\varepsilon}}{\partial
x_\alpha}\right)^2=0.\]
\end{prop}
This clearly implies Lemma \ref{lem 3.11} in this special case. This
proposition can be proved as in \cite[Lemma 4.6]{Wang}. This proof
is by choosing $X=\varphi\psi x_n e_n$ in the stationary condition
\eqref{stationary condition singularly perturbed}, where $\varphi\in
C_0^\infty(B_1^{n-1})$ and $\psi\in C_0^\infty((-1,1))$. For another
proof using the monotonicity formula, see the derivation of
\cite[Eq. (2.11)]{L}.

 With
the help of Proposition \ref{excess small}, we can get the following
quantization result for $\Theta(x)$.
\begin{thm}\label{quantization}
$\Theta(x)/\sigma_0$ equals positive integer
$\mathcal{H}^{n-1}$-a.e. on $\Sigma$.
\end{thm}
To prove this theorem, we need a lemma.
\begin{lem}\label{lem 3.15}
For any $\delta>0$, there exists a $b\in(0,1)$ such that, for all
$\varepsilon$ small,
\[\int_{\mathcal{C}_1\cap\{u_\varepsilon>1-b\}}\frac{\varepsilon}{2}|\nabla u_\varepsilon|^2
+\frac{1}{\varepsilon}W(u_\varepsilon)\leq\delta.\]
\end{lem}
The proof uses the strict convexity of $W$ near $1$, in particular,
\begin{equation}\label{equation for gradient estiamtes}
\Delta\left[\frac{\varepsilon}{2}|\nabla
u_\varepsilon|^2+\frac{1}{\varepsilon}W(u_\varepsilon)\right]\geq\frac{\kappa}{\varepsilon^2}\left[\frac{\varepsilon}{2}|\nabla
u_\varepsilon|^2+\frac{1}{\varepsilon}W(u_\varepsilon)\right],\ \ \
\mbox{in}\ \ \{u_\varepsilon>1-b_1\},
\end{equation}
where $b_1>0$ is small. For more details, see \cite[Corollary
6.4]{Wang}.

\begin{proof}[Proof of Theorem \ref{quantization}]
We still need only to consider the special case where
$\Sigma=\R^{n-1}$ and
$\mu=\Theta\mathcal{H}^{n-1}\lfloor_{\R^{n-1}}$, with $\Theta$ a
constant. We want to prove that $\Theta/\sigma_0$ is a positive
integer.

For $x^\prime\in B_1^{n-1}$, let
\[f_\varepsilon(x^\prime):=\int_{-1}^1\frac{\varepsilon}{2}|\nabla u_\varepsilon(x^\prime,x_n)|^2
+\frac{1}{\varepsilon}W(u_\varepsilon(x^\prime,x_n))\chi_{\{u_\varepsilon>0\}}dx_n.\]
By \eqref{energy bound}, $f_\varepsilon$ are uniformly bounded in
$L^1(B_1^{n-1})$. By the convergence of $\varepsilon|\nabla
u_\varepsilon|^2dx$ etc., $f_\varepsilon$ converges to $\Theta$
weakly in $L^1(B_1^{n-1})$.

Fix a $\psi\in C_0^\infty((-1,1))$ such that $\psi\equiv 1$ in
$(-1/2,1/2)$. Let
\[\tilde{f}_\varepsilon(x^\prime):=\int_{-1}^1\left[\frac{\varepsilon}{2}|\nabla u_\varepsilon|^2
+\frac{1}{\varepsilon}W(u_\varepsilon)\chi_{\{u_\varepsilon>0\}}\right]\psi(x_n)dx_n.\]
By Lemma \ref{lem 3.2} and Lemma \ref{lem 3.4},
\begin{equation}\label{3.7}
\int_{B_1^{n-1}}\big|f_\varepsilon-\tilde{f}_\varepsilon\big|\leq
Ce^{-\frac{1}{C\varepsilon}}.
\end{equation}

By substituting $X=\varphi\psi e_i$ with $\varphi\in
C_0^\infty(B_1^{n-1})$, we see
\[\frac{\partial \tilde{f}_\varepsilon}{\partial x_i}=\sum_{j=1}^n\frac{\partial}{\partial x_j}A_{ij}^\varepsilon
+g_\varepsilon^i,\quad \forall 1\leq i\leq n-1,\] where
\[A_{ij}^\varepsilon:=\int_{-1}^1\varepsilon\frac{\partial u_\varepsilon}{\partial x_i}
\frac{\partial u_\varepsilon}{\partial x_j}\psi(x_n)dx_n,\]
\[g_\varepsilon^i=\int_{-1}^1\varepsilon\frac{\partial u_\varepsilon}{\partial x_i}
\frac{\partial u_\varepsilon}{\partial x_n}\psi^\prime(x_n)dx_n.\]

By Proposition \ref{excess small} and Cauchy inequality, for all
$1\leq i\leq n-1$ and $1\leq j\leq n$, $A_{ij}^\varepsilon$ and
$g^i_\varepsilon$ converges to $0$ in $L^1_{loc}(B_1^{n-1})$. Then
by Allard's strong constancy lemma \cite{Allard},
$\tilde{f}_\varepsilon$ converges to $\Theta$ in
$L^1_{loc}(B_1^{n-1})$, which also holds for $f_\varepsilon$ by
\eqref{3.7}.

By Lemma \ref{excess small} and the weak $L^1$ estimate for
Hardy-Littlewood maximal function, there exists a set
$E^1_\varepsilon$ with $|B_{1/2}^{n-1}\setminus
E^1_\varepsilon|<|B_{1/2}^{n-1}|/4$, such that,
\begin{equation}\label{3.1}
\lim_{\varepsilon\to0}\sup_{r\in(0,1/2)}r^{1-n}\int_{\mathcal{C}_r(x^\prime)}\varepsilon\sum_{\alpha=1}^{n-1}
\left(\frac{\partial u_{\varepsilon}}{\partial x_\alpha}\right)^2=0,
\quad \forall x^\prime\in E^1_\varepsilon.
\end{equation}

By Lemma \ref{lem 3.15}, for any $\delta>0$, there exists a
$b\in(0,1)$ and a set $E^2_\varepsilon$ with
$|B_{1/2}^{n-1}\setminus E^2_\varepsilon|<|B_{1/2}^{n-1}|/4$, such
that
\begin{equation}\label{3.2}
\limsup_{\varepsilon\to0}\sup_{r\in(0,1/2)}r^{1-n}\int_{\mathcal{C}_r(x^\prime)
\cap\{u_\varepsilon>1-b\}}\frac{\varepsilon}{2}|\nabla
u_\varepsilon|^2 +W(u_\varepsilon)\leq C\delta, \quad \forall
x^\prime\in E^2_\varepsilon.
\end{equation}

By applying the weak $L^1$ estimate for Hardy-Littlewood maximal
function to $|f_\varepsilon-\Theta|$, we get a set $E^3_\varepsilon$
with $|B_{1/2}^{n-1}\setminus E^3_\varepsilon|<|B_{1/2}^{n-1}|/4$,
such that
\begin{equation}\label{3.3}
\lim_{\varepsilon\to0}\sup_{r\in(0,1/2)}r^{1-n}\int_{B_r(x^\prime)}|f_\varepsilon(x^\prime)-\Theta|=0,
\quad \forall x^\prime_\varepsilon\in E^3_\varepsilon.
\end{equation}

Now choose $x^\prime_\varepsilon\in E^1_\varepsilon\cap
E^2_\varepsilon\cap E^3_\varepsilon$. For any
$x_\varepsilon:=(x^\prime_\varepsilon,
x^n_\varepsilon)\in\partial\{u_\varepsilon>0\}$,
\[v^\varepsilon(x):=u_\varepsilon(x_\varepsilon+\varepsilon x)\]
converges to a limit $v^\infty$ in $C_{loc}(\R^n)\cap
H^1_{loc}(\R^n)$ (by the a priori estimates in \cite{AC}). By
\eqref{3.1}, $v^\infty$ depends only on the $x_n$ variable, hence
equals the one dimensional profile $g$. Thus for all $\varepsilon>0$
small, $v^\varepsilon(0,x_n)>0$ in $(0,g^{-1}(b))$ and
\begin{equation}\label{3.6}
\lim_{\varepsilon\to0}\int_0^{g^{-1}(b)}\frac{1}{2}\left(\frac{\partial
v^\varepsilon}{\partial
x_n}\right)^2+W(v^\varepsilon)=\int_0^{g^{-1}(b)}\frac{1}{2}\left(g^\prime\right)^2+W(g)=\sigma_0+o_b(1).
\end{equation}

Assume that
$\Pi^{-1}(x^\prime_\varepsilon)\cap\partial\{u_\varepsilon>0\}$
consists $N_\varepsilon$ points, $t^i_\varepsilon$, $1\leq i\leq
N_\varepsilon$. By the analysis above, for all $\varepsilon$ small,
$u_\varepsilon>1-b$ or $u_\varepsilon=0$ outside
\[G_\varepsilon:=B_\varepsilon^{n-1}(x^\prime_\varepsilon)\times\cup_{1\leq i\leq
N_\varepsilon}(t^i_\varepsilon-M\varepsilon,t^i_\varepsilon+M\varepsilon),\]
where $M$ is a constant depending only on $b$.

Then
\begin{eqnarray*}
&&\varepsilon^{1-n}\int_{C_\varepsilon(x^\prime_\varepsilon)}e_\varepsilon(u_\varepsilon)\\
&=&\varepsilon^{1-n}\int_{G_\varepsilon}e_\varepsilon(u_\varepsilon)+\varepsilon^{1-n}\int_{C_\varepsilon(x^\prime_\varepsilon)\setminus
G_\varepsilon}e_\varepsilon(u_\varepsilon)\\
&=&\sum_{i=1}^{N_\varepsilon}\varepsilon^{1-n}\int_{B_\varepsilon^{n-1}(x^\prime_\varepsilon)}
\int_{-M}^M\left(\frac{1}{2}|\nabla
v^\varepsilon|^2+W(v^\varepsilon)\chi_{\{v^\varepsilon>0\}}\right)+O(\delta)\\
&=&N_\varepsilon
\left(\sigma_0+o_\varepsilon(1)+o_b(1)\right)+O(\delta).\quad
\mbox{(by \eqref{3.6})}
\end{eqnarray*}
On the other hand, by \eqref{3.3},
\[\lim_{\varepsilon\to0}
\varepsilon^{1-n}\int_{C_\varepsilon(x^\prime_\varepsilon)}e_\varepsilon(u_\varepsilon)=\Theta.\]
Hence
\[\lim_{\varepsilon\to0}N_\varepsilon=\frac{\Theta}{\sigma_0}+
o_b(1)+O(\delta).\] The last two terms can be made arbitrarily
small. Then because $N_\varepsilon$ is a positive integer, it must
be constant for all $\varepsilon$ small, which also equals
$\Theta/\sigma_0$.
\end{proof}

Define the varifold $V$ by
\[<V,\varphi>:=\int_{\Sigma}\varphi(x,T_x\Sigma)\Theta(x)d\mathcal{H}^{n-1},\]
for $\varphi\in C_0^\infty(B_1\times \mathbb{RP}^n)$. (We view the
space of hyperplanes of $\R^n$ as the projective space
$\mathbb{RP}^n$.) By passing to the limit in the stationary
condition for $u_\varepsilon$ and noting Lemma \ref{lem 3.11}, we
obtain
\begin{lem}
$V$ is stationary.
\end{lem}

Finally, we would like to compare this convergence theory with the
$\Gamma$-convergence theory. Let
\[w_\varepsilon(x):=\Phi(u_\varepsilon(x))=\int_0^{u_\varepsilon(x)}\sqrt{2W(t)}dt.\]
Then
\begin{eqnarray*}
\int_{B_1}|\nabla
w_\varepsilon|&=&\int_{B_1}\sqrt{2W(u_\varepsilon)}|\nabla
u_\varepsilon|\\
&\leq&\int_{B_1}\frac{1}{2}|\nabla
u_\varepsilon|^2+\frac{1}{\varepsilon}W(u_\varepsilon)\chi_{\{u_\varepsilon>0\}}\leq
C.
\end{eqnarray*}
Since $0\leq w_\varepsilon\leq \int_0^1\sqrt{2W(t)}dt$, it is
uniformly bounded in $BV_{loc}(B_1)$. Then up to a subsequence
$w_\varepsilon$ converges in $L^1_{loc}(B_1)$ to a function
$w_\infty\in BV_{loc}(B_1)$.

By extending $\Phi$ suitably to (-1,1), there exists a continuous
inverse of it. Then $u_\varepsilon=\Phi^{-1}(w_\varepsilon)$
converges to $\Phi^{-1}(w_\infty)$ in $L^1_{loc}(\R^n)$. Since
\[\int_{B_1}W(u_\varepsilon)\chi_{\{u_\varepsilon>0\}}\leq C\varepsilon,\]
$u_\varepsilon\to0$ or $1$ a.e. in $B_1$. Hence there exists a
measurable set $\Omega_\infty$ such that
\[u_\varepsilon\to\chi_{\Omega_\infty},\quad \mbox{in}\ L^1_{loc}(\R^n).\]
Because $w_\infty=(\int_0^1\sqrt{2W(t)}dt)\chi_{\Omega_\infty}$,
$\chi_{\Omega_\infty}\in BV_{loc}(\R^n)$.

For minimizers, combining the above two approaches gives
\begin{prop}
If $u_\varepsilon$ are minimizers, then
$\Sigma=\partial\Omega_\infty$ and
$\mu=\sigma_0\mathcal{H}^{n-1}\lfloor_{\partial\Omega_\infty}$.
Moreover, $\Omega_\infty$ is a set with minimal perimeter.
\end{prop}
The first claim can be proved by the method of cut and paste, i.e.
constructing suitable comparison functions. The second one follows
from the standard $\Gamma$-convergence theory (see \cite{Modica 2}).

\section{Improvement of flatness}
\setcounter{equation}{0}

Now we return to the study of entire solutions of \eqref{equation}.
Let $u$ be a local minimizer of the functional \eqref{functional} in
$\R^n$. For $\varepsilon\to0$, we can apply results in the previous
section to study the convergence of the blowing down sequence
\[u_\varepsilon(x)=u(\varepsilon^{-1}x).\]

In this section we assume the blowing gown limit
$\Omega_\infty=\{x_n>0\}$ for a subsequence $\varepsilon_i\to0$.
(However, at this stage we do not know whether this limit depends on
the subsequence of $\varepsilon\to0$.) Note that this is always true
if
 $n\leq 7$, by Bernstein theorem.

The following quantity will play an important role in our analysis.
\begin{defi}[\bf Excess]
Let $P$ be an $(n-1)$-dimensional hyperplane in $\R^n$ and $e$ one
of its unit normal vector, $B_r^{n-1}(x)\subset P$ an open ball and
$\mathcal{C}_r(x)=B_r^{n-1}(x)\times(-1,1)$ the cylinder over
$B_r(x)$. The excess of $u_\varepsilon$ in $\mathcal{C}_r(x)$ with
respect to $P$ is
\begin{equation}\label{excess}
E(r;x,u_\varepsilon,P):=r^{1-n}\int_{\mathcal{C}_r(x)}\left[1-\left(\nu_\varepsilon\cdot
e\right)^2\right]\varepsilon|\nabla u_\varepsilon|^2.
\end{equation}
\end{defi}
Here $\nu_\varepsilon=\nabla u_\varepsilon/|\nabla u_\varepsilon|$
when $|\nabla u_\varepsilon|\neq 0$, otherwise we take it to be an
arbitrary unit vector.

In Proposition \ref{excess small}, we have shown that if the blowing
down limit (of $u_\varepsilon$) is a hyperplane, then the excess
with respect to this hyperplane converges to $0$.

Our main tool to prove Theorem \ref{main result 2} is the following
decay estimate. As in \cite{Wang}, we state this theorem for a
general stationary critical point of \eqref{functional singular
perturbed}, not necessarily a minimizer.
\begin{thm}[\bf Tilt-excess decay]\label{thm tilt excess decay}
For any given constant $b\in(0,1)$, there exist five universal
constants $\delta_0,\tau_0,\varepsilon_0>0$, $\theta\in(0,1/4)$ and
$K_0$ large so that the following holds. Let $u_\varepsilon$ be a
solution of \eqref{equation} with $\varepsilon\leq \varepsilon_0$ in
$B_4$, satisfying the Modica inequality,
$0\in\partial\{u_\varepsilon>0\}$, and
\begin{equation}\label{close to plane 0}
4^{-n}\int_{\mathcal{B}_4}\frac{\varepsilon}{2}|\nabla
u_\varepsilon|^2+\frac{1}{\varepsilon}W(u_\varepsilon)\chi_{\{u_\varepsilon>0\}}\leq
\left(1+\tau_0\right)\sigma_0\omega_n.
\end{equation}
Suppose the excess with respect to $\R^{n-1}$
\begin{equation}\label{small excess 0}
\delta_\varepsilon^2:=E(2;0,u_\varepsilon,\R^{n-1})\leq\delta_0^2,
\end{equation}
where $\delta_\varepsilon\geq K_0\varepsilon$. Then there exists
another plane $P$, such that
\begin{equation}\label{excess decay}
E(\theta;0,u_\varepsilon,
P)\leq\frac{\theta}{2}E(2;0,u_\varepsilon,\R^n).
\end{equation}
Moreover, there exists a universal constant $C$ such that
\begin{equation}\label{control on the direction}
\|e-e_{n+1}\|\leq CE(2;0,u_\varepsilon,\R^n)^{1/2},
\end{equation}
 where $e$ is
the unit normal vector of $P$ pointing to the above.
\end{thm}
The proof of this theorem is similar to the one in \cite{Wang}. It
is mainly divided into four steps:
\begin{enumerate}
\item [{\bf Step 1.}] $\partial\{u_\varepsilon>0\}$ and $\{u_\varepsilon=t\}$ (for $t\in(0,1-b)$ with $b>0$
fixed) can be represented by Lipschitz graphs over $\R^{n-1}$,
$x_n=h_\varepsilon^t(x^\prime)$, except a bad set of small measure
(controlled by $E(2;0,u_\varepsilon,\R^{n-1})$). This can be
achieved by the weak $L^1$ estimate for Hardy-Littlewood maximal
functions, as in the proof of Theorem \ref{quantization}.

\item [{\bf Step 2.}] By writing the excess using the $(x^\prime,t)$
coordinates ($t$ as in Step 1), $h_\varepsilon^t/\delta_\varepsilon$
are uniformly bounded in $W^{1,2}_{loc}(B_1^{n-1})$. Then we can
assume that they converge weakly to a limit $h_\infty$. Here we need
the assumption $\delta_\varepsilon\gg \varepsilon$ to guarantee the
limit is independent of $t$.

\item [{\bf Step 3.}] By choosing $X=\varphi\psi e_n$ in the
stationary condition \eqref{stationary condition singularly
perturbed}, where $\varphi\in C_0^\infty(B_1^{n-1})$ and $\psi\in
C_0^\infty((-1,1))$, and then passing to the limit, it is shown that
$h_\infty$ is harmonic in $B_1^{n-1}$.

\item [{\bf Step 4.}] By choosing $X=\varphi\psi x_ne_n$ in the
stationary condition \eqref{stationary condition singularly
perturbed} and then passing to the limit, it is shown that (roughly
speaking) $h_\varepsilon^t/\delta_\varepsilon$ converges strongly in
$W^{1,2}_{loc}(B_1^{n-1})$. The tilt-excess decay estimate then
follows from some basic estimates on harmonic functions.
\end{enumerate}

As in \cite{Wang}, using this theorem we can prove the following
estimate.
\begin{lem}
There exists a unit vector $e_\infty$ and a universal constant
$C(n)$ such that
\begin{equation}\label{4.1}
\int_{\mathcal{B}_R(x)}\left[1-\left(\nu\cdot
e_\infty\right)^2\right]|\nabla u|^2\leq C(n)R^{n-2},\ \ \forall \
x\in\R^n,\ R>1.
\end{equation}
\end{lem}
Note that here the exponent $n-2<n-1$, which is the energy growth
order of $u$ (see Corollary \ref{clearing out 1} and Proposition
\ref{energy growth}). The blowing down analysis in Section 3 only
gives
\[\int_{\mathcal{B}_R(x)}\left[1-\left(\nu\cdot
e_{R,x}\right)^2\right]|\nabla u|^2=o(R^{n-1}),\] where the unit
vector $e_{R,x}$ may also depend on $x$ and $R$. However, by
iterating Theorem \ref{thm tilt excess decay} on balls of the form
$B_{\theta^{-i}}(x)$, we not only get the decay of the excess on
these balls, but also get a control on the variation of
$e_{x,\theta^{-i}}$ (through the estimate \eqref{control on the
direction}).

Still as in \cite{Wang}, \eqref{4.1} implies the uniqueness of the
blowing down limit constructed in the previous section.

Next consider the distance type function
\[\Psi(x):=g^{-1}\circ u.\]
It satisfies
\[
-\Delta\Psi=f(\Psi)(1-|\nabla\Psi|^2),\quad \mbox{in}\
\{\Psi>0\}=\{u>0\},
\]
where $f(t):=\frac{W^\prime(g(t))}{\sqrt{2W(g(t))}}$.

By the vanishing viscosity method, as $\varepsilon\to0$,
\[\Psi_\varepsilon(x):=\varepsilon\Psi(\varepsilon^{-1}x)\]
converges to a limit $\Psi_\infty$ uniformly on any compact set of
$\R^n$, and in $C^1$ on any compact set of $\{\Psi_\infty>0\}$.
Moreover, in $\{\Psi_\infty>0\}$, $\Psi_\infty$ is a viscosity
solution to the eikonal equation
\[|\nabla\Phi_0|^2-1=0.\]

By definition, we can show that $\{\Psi_\infty>0\}=\Omega_\infty$.
Using the estimate \eqref{4.1} we know $\Psi_\infty$ depends only on
the $e_\infty$ direction. Hence after suitable rotation
$\Psi_\infty=x_n^+$.

The $C^1$ convergence of $\Psi_\varepsilon$ then implies that
$\nabla\Psi$ is arbitrarily close to $e_n$, as far as $u$ is close
enough to $1$, in a uniform manner. This then enables us to apply
the sliding method to deduce that $u$ depends only the $x_n$
variable, hence finish the proof of Theorem \ref{main result 2}.

\section{Serrin's overdetermined problem}
\setcounter{equation}{0}

In this section we assume that $u$ is a solution of \eqref{equation
0} satisfying the monotonicity condition \eqref{monotinicity in one
direction}, where $W$ is a double well potential satisfying the
hypothesis {\em (W1-5)}.

We first need a technical lemma for the application below.
\begin{lem}\label{lem 5.01}
Let $u$ be a $C^2$ solution of
\[\Delta u=W^\prime(u),\quad\mbox{in}\ \R^n,\]
satisfying $0<u\leq 1$. Then $u\equiv 1$.
\end{lem}
For a proof see \cite[Section 4.1]{KLP}.

\begin{lem}\label{lem 5.1}
For any $x^\prime\in\R^{n-1}$,
\[\lim_{x_n\to+\infty}u(x^\prime,x_n)=1,\]
and $u(x^\prime,-x_n)=0$ for all $x_n>0$ large.
\end{lem}
\begin{proof}
By a contradiction argument, we can show that
\[\lim_{t\to\pm\infty}\mbox{dist}\left((x^\prime,te_n),\partial\Omega\right)=+\infty.\]
Thus for any $R>0$ and $t>0$ large,
$B_R(x^\prime,-te_n)\subset\Omega^c$. In particular,
$u(x^\prime,-te_n)=0$ for all $t$ large.

By the same reasoning and standard elliptic estimates, as
$t\to+\infty$,
\[u^t(\cdot)=u((x^\prime,te_n)+\cdot)\]
converges in $C^2_{loc}(\R^n)$ to a limit function $u_\infty$, which
is a positive solution of
\[\Delta u_\infty=W^\prime(u_\infty)\]
in $\R^n$. Since $0<u_\infty\leq 1$, by the previous lemma,
$u_\infty\equiv 1$.
\end{proof}

\begin{lem}
\label{upper}
On $\partial\Omega$, $|\nabla u|\geq\sqrt{2W(0)}$.
\end{lem}
\begin{proof}
By the same proof as in the previous lemma, for any $R>0$ there
exists a $t_0>0$ such that, for all $t\geq t_0$, the ball
$B_R(0,t)\subset\Omega$.

Let $v^R$ be the unique radial solution of
\begin{equation*}
 \left\{\begin{aligned}
&\Delta v^R=W^\prime(v^R),\ \ \ \mbox{in}\ B_R,\\
&v^R>0,\ \ \ \mbox{in}\ B_R,\\
&v^R=0,\ \ \ \mbox{on}\ \partial B_R.
                          \end{aligned} \right .
\end{equation*}

For any $x$ and $R>0$, denote $v^R_x:=v^R(\cdot-x)$.

Since $\sup _{B_R}v^R<1$, if $t$ is large enough, $v^R_{te_n}<u$ in
$B_R(te_n)$. Let
\[t^\ast:=\inf\{t: B_R(te_n)\subset\Omega\}.\]
Then $B_R(t^\ast e_n)$ is tangent to $\partial\Omega$ at some point
$x_0$.

By \cite[Lemma 3.1]{BCN}, for all $t\geq t^\ast$, $u>v^R_{te_n}$ in
$B_R(te_n)$. The Hopf lemma implies that
\begin{equation}\label{5.1}
|\nabla u(x_0)|=\frac{\partial u}{\partial
\nu}(x_0)\geq\frac{\partial v^R_{t^\ast e_n}}{\partial\nu}(x_0).
\end{equation}
Here $\nu$ is the upward unit normal vector of $\partial\Omega$.
Because $B_R(t^\ast e_n)$ is tangent to $\partial\Omega$ at $x_0$,
we have
\begin{equation}\label{5.2}
\frac{\partial v^R_{t^\ast e_n}}{\partial\nu}(x_0) =-\frac{\partial
v^R}{\partial r}(Re_n)=|\nabla v(Re_n)|.
\end{equation}
On the other
hand, as $R\to+\infty$, $v_R(Re_n+\cdot)$ converges to a positive
solution of
\begin{equation*}
 \left\{\begin{aligned}
&\Delta v^\infty=W^\prime(v^\infty),\ \ \ \mbox{in}\ \R^n_+,\\
&v^\infty>0,\ \ \ \mbox{in}\ \R^n_+,\\
&v^\infty=0,\ \ \ \mbox{on}\ \partial \R^n_+.
                          \end{aligned} \right .
\end{equation*}
Because $v^R$ is radial, $v^\infty$ depends only on the $x_n$
variable. (In fact, to deduce this we do not need the radial
symmetry of $v^R$, see \cite{BCN} and references therein.) Hence it
satisfies the ODE version of \eqref{equation 0} and the conservation
relation \eqref{first integral}. In particular,
\[\sqrt{2W(0)}=|\nabla v^\infty(0,0)|=\lim_{R\to+\infty}|\nabla v_R(Re_n)|.\]
Combining this with \eqref{5.1} and \eqref{5.2} we finish the proof.
\end{proof}

\begin{lem}
\label{lower}
On $\partial\Omega$, $|\nabla u|\leq\sqrt{2W(0)}$.
\end{lem}
\begin{proof}
As in the previous lemma, for any $R>0$ we find a ball
$B_R(0,-t^\ast e_n)\subset\Omega^c$  tangent to $\partial\Omega$ at
a point $x_0$.

In $B_{2R}(0,-t^\ast e_n)\setminus B_R(0,-t^\ast e_n)$, by the Kato
inequality,
\[\Delta u\geq W^\prime(u).\]
Clearly the constant function $1$ is a sup solution of this equation
in $B_{2R}(0,-t^\ast e_n)\setminus B_R(0,-t^\ast e_n)$. Because
$1>u$, by the standard sup-sub solution method, there exists a
solution $u^R>u$ in $B_{2R}(0,-t^\ast e_n)\setminus B_R(0,-t^\ast
e_n)$ satisfying $u^R=0$ on $\partial B_R(0,-t^\ast e_n)$ and
$u^R=1$ on $\partial B_{2R}(0,-t^\ast e_n)$. Then the Hopf lemma
implies that
\[|\nabla u(x_0)|=\frac{\partial u}{\partial\nu}(x_0)\leq\frac{\partial u^R}{\partial\nu}(x_0)
=|\nabla u^R(x_0)|\leq\sqrt{2W(0)}+o_R(1).\] Here the last identity
follows from the same argument as in the previous lemma.
\end{proof}

Combining Lemma (\ref{upper}) and Lemma (\ref{lower}) we obtain the
proof of Theorem \ref{reduction to free boundary}. Note that up to
now we have not used the monotonicity condition \eqref{monotinicity
in one direction}.  However, this condition is crucial for the
following result.
\begin{lem}
$u$ is a local minimizer of the functional \eqref{functional} in
$\R^n$.
\end{lem}
\begin{proof}
Assume by the contrary, there exists a ball $B_R(x_0)$ such that $u$
is not a minimizer of the functional \eqref{functional} in this ball
(under the same boundary condition as $u$). Let $v$ be such a
minimizer.

For any $t\in\R$, consider
\[u^t(x):=u(x+te_n).\]
By Lemma \ref{lem 5.1}, for all $t$ large, $u^t>0$ and $u^t>v$ in
$B_R(x_0)$. Let
\[t_+:=\inf\{t: u^s\geq v \ \mbox{on}\ \overline{B_R(x_0)},\ \forall s>t\}.\]
We claim that $t_+=0$.

Assume by the contrary, $t_+>0$. By definition and continuity,
$u^{t_+}\geq  v$ on $\overline{B_R(x_0)}$. Moreover, by the
monotonicity of $u$, $u^{t_+}\neq v$ on $\partial B_R(x_0)$. Then by
the strong maximum principle and Hopf lemma, $\{v>0\}\cap B_R(x_0)$
is strictly contained in $\{u^{t_+}>0\}\cap B_R(x_0)$ and $u^{t_+}>
v$ strictly on $\overline{\{v>0\}\cap B_R(x_0)}$.

By continuity, there exists an $\epsilon>0$ such that, for all
$t\in(t_+-\epsilon,t_+]$, $u^t\geq v$ on $\overline{B_R(x_0)}$. This
contradicts the definition of $t_+$. Hence we must have $t_+=0$,
which implies that $u\geq v$ on $\overline{B_R(x_0)}$.

Because for all $t>0$ large, $u^{-t}\equiv 0$ on
$\overline{B_R(x_0)}$, we can also slide from below. This gives
$u\leq v$ on $\overline{B_R(x_0)}$. Hence $u\equiv v$ is the unique
minimizer of the energy functional \eqref{functional}.
\end{proof}

With this lemma in hand, we can perform the blowing down analysis as
in the one phase free boundary problem. By the proof of
\cite[Theorem 2.4]{Savin}, we can show (using the notations in
Section 3)
\begin{lem}
If $n\leq 8$, $\Omega_\infty$ is an half space.
\end{lem}

With this lemma in hand, we can use the method in the previous
section to show that $u$ is one dimensional, thus completing the
proof of Theorem \ref{main result}.

Finally we prove Theorem \ref{main result 3}, \ref{main result 5}
and \ref{main result 4}.
\begin{proof}[Proof of Theorem \ref{main result 3} and  Theorem \ref{main result 5}]
By \cite{BCN}, $u$ satisfies the monotonicity condition
\eqref{monotinicity in one direction} in $\Omega$. As in the
previous proof, $|\nabla u|=\sqrt{2W(0)}$ on $\partial\Omega$ and
$u$ is a local minimizer for the functional \eqref{functional}. Then
we can perform the blowing down analysis as before.

If $\varphi $ is globally Lipschitz, the blowing down limit
$\Omega_\infty$ is still the epigraph associated to a Lipschitz
function $\varphi_\infty$ defined on $\R^{n-1}$. Since
$\Omega_\infty$ has minimal perimeter, $\varphi_\infty$ satisfies
the minimal surface equation. By \cite[Theorem 17.5]{Giusti},
$\varphi_\infty$ must be an affine function. In other words,
$\Omega_\infty$ is an half space. Then we deduce that $\Omega$ is an
half space and $\varphi$ is an affine function. However, this
contradicts the coerciveness of $\varphi$.
\end{proof}

\begin{proof}[Proof of Theorem \ref{main result 4}]
In $\R^2$, by Remark \ref{rmk 1}, $\{u=0\}$ is a convex set. Hence
the function $\varphi$ is concave.

First,
\begin{eqnarray*}
-\int_{B_R(0)\cap\{u>0\}}W^\prime(u)&=&-\int_{B_R(0)\cap\{u>0\}}\Delta
u\\
&=&-\int_{\partial B_R(0)\cap\{u>0\}}\frac{\partial u}{\partial
r}+\int_{B_R(0)\cap\partial\{u>0\}}|\nabla u|\\
&\leq &CR.\quad \mbox{(by the Lipschitz bound on $u$)}
\end{eqnarray*}
By hypothesis on $W$, $-W^\prime\geq cW$ on $(\gamma, 1)$. Thus we
obtain
\begin{equation}\label{5.3}
\int_{B_R(0)\cap\{u>\gamma\}}W(u)\leq CR.
\end{equation}

Next, as in the proof of \cite[Theorem 1.2, (b)]{BCN}, there exists
an $M>0$ so that
\begin{equation}\label{5.5}
\{u<\gamma\}\subset\{x: dist(x,\partial\{u>0\})<M\}.
\end{equation}
 By the
convexity of $\partial\{u>0\}$,
\[\big|\{x:
dist(x,\partial\{u>0\})<M\}\cap B_R(0)\big|\leq CR.\] Thus
\begin{equation}\label{5.4}
\int_{B_R(0)\cap\{0<u<\gamma\}}W(u)\leq CR.
\end{equation}

Combining \eqref{5.3}, \eqref{5.4} and the Modica inequality we see
\[\int_{B_R(0)}\frac{1}{2}|\nabla u|^2+W(u)\chi_{\{u>0\}}\leq CR.\]
With this bound in hand, we can perform the blowing down analysis
using results in Section 3. In particular, we obtain a stationary
integer $1$-rectifiable varifold $V$ from the sequence
\[u_\varepsilon(x):=u(\varepsilon^{-1}x).\]

$V$ has the following form: there are finitely many positive
integers $n_i$ and unit vectors $e_i$ with the corresponding rays
$L_i:=\{te_i: t>0\}$, such that
\begin{equation}\label{5.6}
V=\sum_i n_i[L_i],
\end{equation}
where $[L_i]$ is the standard varifold associated to $L_i$.

Because $V$ is stationary, we have the following balancing formula:
\begin{equation}\label{5.7}
\sum_i n_i e_i=0.
\end{equation}

 On the other hand, from the convexity of $\{u=0\}$ it is clear
that as $\varepsilon\to0$, $\varepsilon \{u>0\}$ converges to a
limit $\Omega_\infty$ in the Hausdorff distance, with
$\R^2\setminus\Omega_\infty$ convex. Moreover, by assuming
$0\in\partial\{u=0\}$, $\R^2\setminus\Omega_\infty\subset\{u=0\}$.
Hence for all $\varepsilon>0$, $u_\varepsilon=0$ on
$\R^2\setminus\Omega_\infty$. By \eqref{5.5}, $u_\varepsilon\to1$
a.e. in $\Omega_\infty$. This then implies that the support of $V$
lies in $\partial\Omega_\infty$. Combining \eqref{5.7} and the
convexity of $\R^2\setminus\Omega_\infty$, $\Omega_\infty$ must be
an half plane.

What we have proved says that, the limit cone (at infinity) of the
concave curve $\{x_2=\varphi(x_1)\}$ is a line. By convexity, this
implies that $\{x_2=\varphi(x_1)\}$ itself is a line.

 Finally, there are many ways to deduce that $u$
is one dimensional, see for example \cite{BCN} again.
\end{proof}

\bigskip

\noindent {\bf Acknowledgments.} K. Wang is supported by NSF of
China No. 11301522. J. Wei is partially supported by NSERC of
Canada.

\end{document}